\documentclass[3p,12pt]{elsarticle}
\makeatletter\if@twocolumn\PassOptionsToPackage{switch}{lineno}\else\fi\makeatother

      \makeatletter
\usepackage{wrapfig}
\newcounter{aubio}
\long\def\bioItem{%
\@ifnextchar[{\@bioItem}{\@@bioItem}}

\long\def\@bioItem[#1]#2#3{
 \stepcounter{aubio}
 \expandafter\gdef\csname authorImage\theaubio\endcsname{#1}
 \expandafter\gdef\csname authorName\theaubio\endcsname{#2}
 \expandafter\gdef\csname authorDetails\theaubio\endcsname{#3}
}

\long\def\@@bioItem#1#2{
 \stepcounter{aubio}
 \expandafter\gdef\csname authorName\theaubio\endcsname{#1}
 \expandafter\gdef\csname authorDetails\theaubio\endcsname{#2}
}

\newcommand{\checkheight}[1]{%
  \par \penalty-100\begingroup%
  \setbox8=\hbox{#1}%
  \setlength{\dimen@}{\ht8}%
  \dimen@ii\pagegoal \advance\dimen@ii-\pagetotal
  \ifdim \dimen@>\dimen@ii
    \break
  \fi\endgroup}

\def\printBio{%
  \@tempcnta=0
   \loop
     \advance \@tempcnta by 1
     \def\aubioCnt{\the\@tempcnta}
     \setlength{\intextsep}{0pt}%
     \setlength{\columnsep}{10pt}%
     \expandafter\ifx\csname authorImage\aubioCnt\endcsname\relax%
      \else%
       \checkheight{\includegraphics[height=1.25in,width=1in,keepaspectratio]{\csname authorImage\aubioCnt\endcsname}}
        \begin{wrapfigure}{l}{25mm}
         \includegraphics[height=1.25in,width=1in,keepaspectratio]{\csname authorImage\aubioCnt\endcsname}
        \end{wrapfigure}\par
      \fi
     \noindent\textbf{\csname authorName\aubioCnt\endcsname}\csname authorDetails\aubioCnt\endcsname \par\bigskip
      \ifnum\@tempcnta < \theaubio
   \repeat
   }
\makeatother

\usepackage{tabulary,xcolor}
\usepackage{amsfonts,amsmath,amssymb}
\usepackage[T1]{fontenc}
\makeatletter
\let\save@ps@pprintTitle\ps@pprintTitle
\def\ps@pprintTitle{\save@ps@pprintTitle\gdef\@oddfoot{\footnotesize\itshape \null\hfill\today}}
\def\hlinewd#1{%
  \noalign{\ifnum0=`}\fi\hrule \@height #1%
  \futurelet\reserved@a\@xhline}

\AtBeginDocument{\ifNAT@numbers \biboptions{sort&compress}\fi}
\makeatother

\usepackage{ifluatex}
\ifluatex
\usepackage{fontspec}
\defaultfontfeatures{Ligatures=TeX}
\usepackage[]{unicode-math}
\unimathsetup{math-style=TeX}
\else
\usepackage[utf8]{inputenc}
\fi
\ifluatex\else\usepackage{stmaryrd}\fi

\usepackage{url,multirow,morefloats,floatflt,cancel,tfrupee}
\makeatletter

\AtBeginDocument{\@ifpackageloaded{textcomp}{}{\usepackage{textcomp}}}
\makeatother
\usepackage{colortbl}
\usepackage{xcolor}
\usepackage{pifont}
\usepackage[nointegrals]{wasysym}
\makeatletter

\def\mcWidth#1{\csname TY@F#1\endcsname+\tabcolsep}

\def\cAlignHack{\rightskip\@flushglue\leftskip\@flushglue\parindent\z@\parfillskip\z@skip}
\def\rAlignHack{\rightskip\z@skip\leftskip\@flushglue \parindent\z@\parfillskip\z@skip}

\usepackage{ifxetex}
\ifxetex\else\if@twocolumn\usepackage{dblfloatfix}\fi\fi

\AtBeginDocument{
\expandafter\ifx\csname eqalign\endcsname\relax
\def\eqalign#1{\null\vcenter{\def\\{\cr}\openup\jot\m@th
  \ialign{\strut$\displaystyle{##}$\hfil&$\displaystyle{{}##}$\hfil
      \crcr#1\crcr}}\,}
\fi
}

\AtBeginDocument{%
  \@ifpackageloaded{endfloat}%
   {\renewcommand\efloat@iwrite[1]{\immediate\expandafter\protected@write\csname efloat@post#1\endcsname{}}}{\newif\ifefloat@tables}%
}%

\def\BreakURLText#1{\@tfor\brk@tempa:=#1\do{\brk@tempa\hskip0pt}}
\let\lt=<
\let\gt=>
\def\processVert{\ifmmode|\else\textbar\fi}

\@ifundefined{subparagraph}{
\def\subparagraph{\@startsection{paragraph}{5}{2\parindent}{0ex plus 0.1ex minus 0.1ex}%
{0ex}{\normalfont\small\itshape}}%
}{}

\newcommand\role[1]{\unskip}
\newcommand\aucollab[1]{\unskip}

\@ifundefined{tsGraphicsScaleX}{\gdef\tsGraphicsScaleX{1}}{}
\@ifundefined{tsGraphicsScaleY}{\gdef\tsGraphicsScaleY{.9}}{}
\def\checkGraphicsWidth{\ifdim\Gin@nat@width>\linewidth
	\tsGraphicsScaleX\linewidth\else\Gin@nat@width\fi}

\def\checkGraphicsHeight{\ifdim\Gin@nat@height>.9\textheight
	\tsGraphicsScaleY\textheight\else\Gin@nat@height\fi}

\def\fixFloatSize#1{}
\let\ts@includegraphics\includegraphics

\def\inlinegraphic[#1]#2{{\edef\@tempa{#1}\edef\baseline@shift{\ifx\@tempa\@empty0\else#1\fi}\edef\tempZ{\the\numexpr(\numexpr(\baseline@shift*\f@size/100))}\protect\raisebox{\tempZ pt}{\ts@includegraphics{#2}}}}

\AtBeginDocument{\def\includegraphics{\@ifnextchar[{\ts@includegraphics}{\ts@includegraphics[width=\checkGraphicsWidth,height=\checkGraphicsHeight,keepaspectratio]}}}

\DeclareMathAlphabet{\mathpzc}{OT1}{pzc}{m}{it}

\def\URL#1#2{\@ifundefined{href}{#2}{\href{#1}{#2}}}

\def\UrlOrds{\do\*\do\-\do\~\do\'\do\"\do\-}%
\g@addto@macro{\UrlBreaks}{\UrlOrds}

\@ifundefined{quoteAttrib}
	{}
	{}

\@ifundefined{titlequoteAttrib}
	{}{}

\newenvironment{title-quote}
	 {\list{}{\fontsize{10pt}{12pt}\selectfont\leftmargin.5in\itshape\rightmargin\leftmargin}%
  \item\relax}
  {\endlist}

\makeatother

\def\R{\mathbb{R}}

\def\bx{\boldsymbol{x}}
\newcommand{\norm}[1]{\left\Vert#1\right\Vert}

\numberwithin{equation}{section}
\newtheorem{thm}{Theorem}[section]
\newtheorem{prob}{Problem}[section]

\newtheorem{lemma}{Lemma}[section]
\newproof{proof}{Proof}
\newproof{pot}{Proof of Theorem \ref{thm2}}

\begin{document}

\begin{frontmatter}
	
\title{Zastavnyi Operators and Positive Definite Radial Functions}

\author[aff63c5b8f064b9863807f549819f278608]{Tarik Faouzi}
\ead{tfaouzi@ubiobio.cl}
\author[aff980d40830c7f8bb81f02e241db520037]{Emilio Porcu \corref{contrib-49136fc397578fcb577fd032b6140a3b}}
\ead{emilio.porcu@newcastle.ac.uk}\cortext[contrib-49136fc397578fcb577fd032b6140a3b]{Corresponding author.}
\author[aff974d3bf2cad98dbdf933dc09b59dbb8d]{Moreno Bevilacqua}
\ead{Moreno.bevilacqua@uv.cl}
\author[aff2d0cc33cd806ae5c9cbf7f7b3f92f9ad]{Igor Kondrashuk}
\ead{igor.kondrashuk@ubiobio.cl}

\address[aff63c5b8f064b9863807f549819f278608]{ Department of Statistics\unskip, Applied Mathematics Research Group
    \unskip, University of Bio Bio\unskip, Chile. \\}
\address[aff980d40830c7f8bb81f02e241db520037]{School of Mathematics and Statistics\unskip, University of Newcastle\unskip, UK\unskip. Department of Mathematics and University of Atacama\unskip, Copiap\'o\unskip, and Millennium Nucleus Center
for the Discovery of Structures
in Complex Data\unskip, Chile. \\
}
  	
\address[aff974d3bf2cad98dbdf933dc09b59dbb8d]{University of Valparaiso\unskip,
Department of Statistics\unskip. Millennium Nucleus Center
for the Discovery of Structures
in Complex Data\unskip, Chile. \\}
  	
\address[aff2d0cc33cd806ae5c9cbf7f7b3f92f9ad]{Grupo de Matem\'atica Aplicada\unskip, Departamento de Ciencias B\'asicas\unskip, Universidad del B \'io-B \'io\unskip, Campus Fernando
May, Av. Andres Bello 720, Casilla 447\unskip, Chill\'an\unskip, Chile \\
   }

\begin{abstract}
we consider a  new operator acting on rescaled weighted differences between two members of  the class  $\Phi_d$ of positive definite radial functions.
In particular, we  study the positive  definiteness  of the operator for  the  Mat{\'e}rn,  Generalized Cauchy
and  Wendland families.
\end{abstract}
\begin{keyword}
   Completely Monotonic\sep Fourier Transforms\sep Positive Definite\sep Radial Functions
\end{keyword}

\end{frontmatter}

\section{Introduction}
Positive definite functions are fundamental to many branches of mathematics as well as probability theory, statistics and machine learning amongst others.
There has been an increasing interest in positive definite functions in $d$-dimensional Euclidean spaces ($d$ is a positive integer throughout), and the reader is referred to  \cite{Dale14}, \cite{Sc38}, \cite{Schaback:2011}, \cite{wu95} and  \cite{Wendland:1995}.

 This paper is concerned with the class $\Phi_d$ of continuous functions $\phi:[0, \infty) \mapsto \R$ such that $\phi(0)=1$ and the function $\bx \mapsto \phi(\| \bx \|)$ is positive definite in $\R^d$. The class $\Phi_d$ is nested, with the strict inclusion relation:
$$ \Phi_1 \supset \Phi_2 \supset \cdots \supset \Phi_d \supset \cdots \supset \Phi_{\infty}:= \bigcap_{d \ge 1} \Phi_d. $$
The classes $\Phi_d$ are convex cones that are closed under product, non-negative linear combinations, and pointwise convergence. Further, for a given member $\phi$ in $\Phi_d$, the rescaled function $\phi(\cdot/\alpha)$ is still in $\Phi_d$ for any given $\alpha$. We make explicit emphasis on this fact because it will be repeatedly used subsequently.

For any nonempty set $A\subseteq\R^d$, we call $C(A)$ the set of continuous functions from $A$ into $\R$.
For $p$ a positive integer, let $\boldsymbol{\theta} \in \Theta \subset \R^p$ and let $\{ \phi(\cdot; \boldsymbol{\theta}), \; \boldsymbol{\theta} \in \Theta \}$ be a parametric family belonging to the class $\Phi_d$. For $\varepsilon \in \R$, $\varepsilon \neq 0$ and $0<\beta_1<\beta_2$ with $\beta_i$, $i=1,2$ two scaling parameters, we define the Zastavnyi operator
$K_{\varepsilon; \boldsymbol{\theta};\beta_2,\beta_1}[\phi]: \Phi_d \mapsto C(\R)$ by
\def\btheta{\boldsymbol{\theta}}
\begin{equation}
\label{zastavnyi1}
K_{\varepsilon; \boldsymbol{\theta};\beta_2,\beta_1}[\phi](t) = \frac{\beta_2^{\varepsilon} \phi \left ( \frac{t}{\beta_2};\btheta\right )-\beta_1^{\varepsilon} \phi \left ( \frac{t}{\beta_1};\btheta\right ) }{\beta_2^{\varepsilon}-\beta_1^{\varepsilon}}, \qquad t \ge 0,
\end{equation}
with $K_{\varepsilon; \boldsymbol{\theta};\beta_2,\beta_1}[\phi](0)=1$. Here, by $\beta_i^{\varepsilon}$ we mean  $\beta_i$ raised to the power of  $\varepsilon$. It can be namely checked that
 $$K_{\varepsilon; \boldsymbol{\theta};\beta_2,\beta_1}[\phi](0) = \frac{\beta_2^{\varepsilon} \phi \left ( \frac{0}{\beta_2};\btheta\right )-\beta_1^{\varepsilon} \phi \left ( \frac{0}{\beta_1};\btheta\right ) }{\beta_2^{\varepsilon}-\beta_1^{\varepsilon}}=\frac{\beta_2^{\varepsilon} -\beta_1^{\varepsilon} }{\beta_2^{\varepsilon}-\beta_1^{\varepsilon}}=1.$$
A motivation for studying positive definiteness of the radial functions $\R^d \ni \mathbf{x} \mapsto K_{\varepsilon; \boldsymbol{\theta};\beta_2,\beta_1}[\phi](\|\mathbf{x}\|)$
comes from the problem of monotonicity of the so--called  microergodic parameter
of specific parametric families \citep{Bevilacqua_et_al:2018,Bevilacqua:2018ab} when studying the asymptotic properties of the maximum likelihood estimation under fixed domain asymptotics.
The operator (\ref{zastavnyi1}) is  a generalization of the operator proposed in \cite{Porcu:Zastavnyi:Xesbaiat}  where $\varepsilon$ is assumed to be positive.
Our problem can be formulated as follows:
\begin{prob}\label{PP}
Let $d$ and $q$ be  positive integers.
Let $\{ \phi(\cdot;\boldsymbol{\theta} ), \; \boldsymbol{\theta} \in \Theta \subset \R^q \}$ be a family of functions belonging to the class  $\Phi_d$.
Find the conditions on $\varepsilon \in \R$, $\varepsilon \neq 0$ and $\boldsymbol{\theta}$, such that $K_{\varepsilon; \boldsymbol{\theta};\beta_2,\beta_1}[\phi]$ as defined through (\ref{zastavnyi1}) belongs to the class $\Phi_n$
for some $n=1,2,\ldots$
 for given $0<\beta_1<\beta_2$.
\end{prob}
We first note that Problem \ref{PP} has at least two  possible solutions. Indeed, direct inspection shows
$$\lim_{\varepsilon \to +\infty} K_{\varepsilon; \boldsymbol{\theta};\beta_2,\beta_1}[\phi](t)  = \phi \left ( \frac{t}{\beta_1};\btheta\right ), \quad
\lim_{\varepsilon \to -\infty} K_{\varepsilon; \boldsymbol{\theta};\beta_2,\beta_1}[\phi](t)  = \phi \left ( \frac{t}{\beta_2};\btheta\right ), \qquad t\geq0,$$
where the convergence is pointwise in $t$.

The  positive definiteness  of  (\ref{zastavnyi1}), assuming $\varepsilon>0$, has been studied  in \cite{Porcu:Zastavnyi:Xesbaiat} when $\phi$ belongs to
the Buhmann class \citep{Buhmann:2001x}. An important special case of the Buhmann class   is the
the  Generalized Wendland family
  \citep{Gneiting:2002b}. For  $\kappa>0$,
   we define the class  ${\cal GW}: [0,\infty)  \to \R$ as:

  \begin{equation}\label{eq:wendland}
 {\cal GW}(t;\kappa,\mu)= \begin{cases}   \frac{\int_{ t}^{1} u(u^2- t^2)^{\kappa-1} (1-u)^{\mu}\,\rm{d}u}{B(2\kappa,\mu+1)}  ,& 0 \leq  t < 1,\\ 0,& t \geq 1, \end{cases}
\end{equation}
  and,  for $\kappa=0$, by continuity we have
    \begin{equation}\label{eq:wendland1}
 {\cal GW}(t;0,\mu)= \begin{cases}  (1-t)^{\mu} ,& 0 \leq  t < 1,\\ 0,& t \geq 1. \end{cases}
\end{equation}
The function ${\cal GW}(t; \kappa,\mu)$ is a member   of the class $\Phi_{d}$ if and only if $\mu\geq 0.5(d+1)+\kappa$ \citep{Zastavnyi:2002}.
\cite{Porcu:Zastavnyi:Xesbaiat} found that if $\phi(\cdot; \btheta)= {\cal GW}(\cdot; {\kappa,\mu})$ and  $\varepsilon>0$
  then
  $K_{\varepsilon; \kappa,\mu;\beta_2,\beta_1}[{\cal GW}](t)$ is positive definite
  if  $\mu \geq (d + 7)/2 + \kappa$ and $\epsilon\geq 2\kappa+1$.





This paper is especially interested  to the solution of Problem \ref{PP}  when considering  two celebrated parametric families:
\begin{description}
\item[The Mat{\'e}rn family.] In this case,  $\phi(\cdot; \btheta)= {\cal M}(\cdot; \nu)$, so that the $\boldsymbol{\theta}= \nu$, a scalar and $\Theta=(0,\infty)$,
with
\begin{equation}
\label{matern} {\cal M}(t; \nu) = \frac{2^{1-\nu}}{\Gamma(\nu)} t^{\nu} {\cal K}_{\nu}(t), \qquad t \ge 0,
\end{equation}
where ${\cal K}_{\nu}$ is the modified Bessel function of the second kind of order $\nu>0$ \citep{Abra:Steg:70}. The functions ${\cal M}(\cdot; \nu)$, $\nu>0$, belong to  the class $\Phi_{\infty}$ \citep{Stein:1999}. 
\item[The Generalized Cauchy family.] In this case  $\phi(\cdot; \btheta)= {\cal C}(\cdot; {\delta,\lambda})$, so that $\boldsymbol{\theta}=(\delta,\lambda)^{\top}$, with $\top$ denoting the transpose operator. Here, $\Theta=(0,2] \times (0,\infty)$, and
\begin{equation} \label{cauchy}
{\cal C}(t; {\delta,\lambda}) = \left ( 1+ t^{\delta} \right )^{-\lambda/\delta}, \qquad t \ge 0.
\end{equation}
 The functions ${\cal C}(\cdot; \delta, \lambda)$ belong to the class $\Phi_{\infty}$ \citep{GneiSchla04}.
\end{description}
Additionally, we provide a solution to Problem \ref{PP} when  $\phi(\cdot; \btheta)= {\cal GW}(\cdot; {\kappa,\mu})$, $\btheta=(\kappa,\mu)^{\top}$, $\Theta=[0,\infty) \times (0,\infty)$, assuming $\varepsilon <0$.

To give an idea of how the operator $K_{\varepsilon; \boldsymbol{\theta};\beta_2,\beta_1}[\phi](t)$ acts on $\phi(t,\boldsymbol{\theta})$ for given $0<\beta_1<\beta_2$
when $\phi$ is the Mat{\'e}rn family,
Figure \ref{fig:fcovt222} (A) compares
$K_{\varepsilon;0.5;\beta_2,\beta_1}[{\cal M}](t)$ with ${\cal M}(t/\beta_1;0.5)$ and
 ${\cal M}(t/\beta_2;0.5)$ when  $\beta_1=0.075$, $\beta_2=0.15$, $\varepsilon=1$ (red line) and $\varepsilon=-2$ (blue line). We note that the behaviour at the origin of $K_{1;0.5;\beta_2,\beta_1}[{\cal M}](t)$  changes drastically with respect to
 the behaviour at the origin of  ${\cal M}(\cdot;0.5)$. Moreover, $K_{-2;0.5;\beta_2,\beta_1}[{\cal M}](t)$ can attain negative values. It turns out  from Theorem \ref{theo11}  that  $K_{1;0.5;\beta_2,\beta_1}[{\cal M}](t) \in \Phi_{\infty}$ and   $K_{-2;0.5;\beta_2,\beta_1}[{\cal M}](t) \in \Phi_2$.

 A similar graphical representation is given in Figure \ref{fig:fcovt222} (B)  when $\phi\equiv {\cal C}$,  the Cauchy family in (\ref{cauchy}). In this case, we consider   $\varepsilon =-0.7, 1.25$, $\delta=0.6$, $\lambda=2.5$,  $\beta_1=0.2$, $\beta_2=0.3$. Note that,under this setting, $K_{-1.25;0.6,2.5;\beta_2,\beta_1}[{\cal C}](t)$ attains negative values as well.
 It turns out  from from Theorem \ref{theo22} $K_{\varepsilon;2.5,0.6;\beta_2,\beta_1}[{\cal C}](t)\in \Phi_{\infty}$ for $\varepsilon=-0.7$ and $1.25$.

\noindent Finally, Figure \ref{fig:fcovt222} (C) compares
$K_{\varepsilon;0,4.5;\beta_2,\beta_1}[{\cal GW}](t)$ with ${\cal GW}(t/\beta_1;0,4.5)$ and
 ${\cal GW}(t/\beta_2;0,4.5)$ with $\beta_2=0.6$, $\beta_1=0.4$ when $\varepsilon=1$ (red line) and $\varepsilon=-2$ (blue line). As for the Mat{\'e}rn case, the behaviour at the origin of $K_{\varepsilon;0,4.5;\beta_2,\beta_1}[{\cal GW}](t)$ changes neatly with respect to
 the behaviour at the origin of  ${\cal GW}(\cdot;0,4.5)$. 
 Moreover $K_{-2;0,4.5;\beta_2,\beta_1}[{\cal GW}](t)$ can reach negative values. It turns out  from Theorem \ref{theo111}  that     $K_{-2;0,4.5;\beta_2,\beta_1}[{\cal GW}](t)$
 belongs to
 $\Phi_2$.

\begin{figure}[h!]
\begin{tabular}{ccc}
   \includegraphics[width=5.2cm, height=6.5cm]{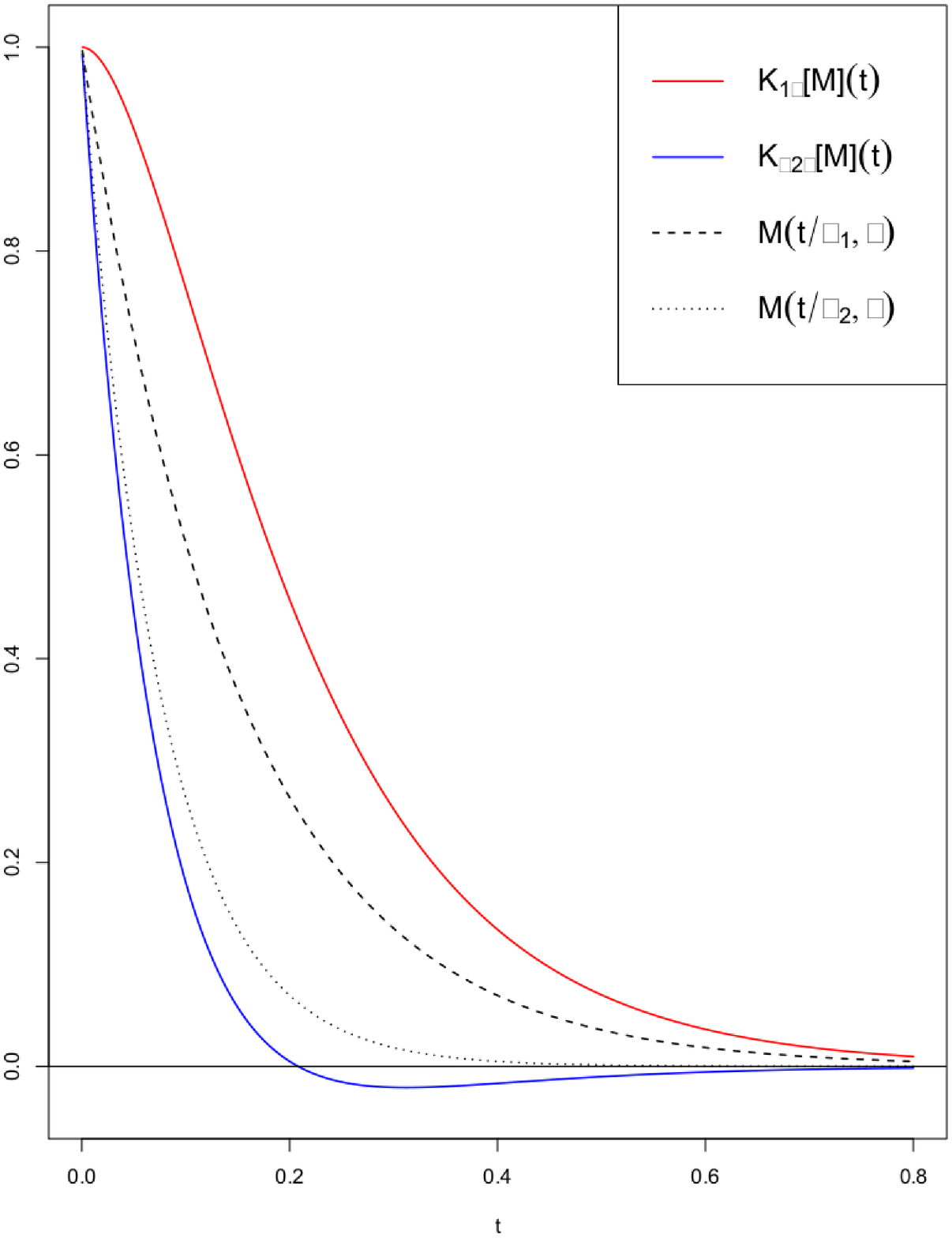}& \includegraphics[width=5.2cm, height=6.5cm]{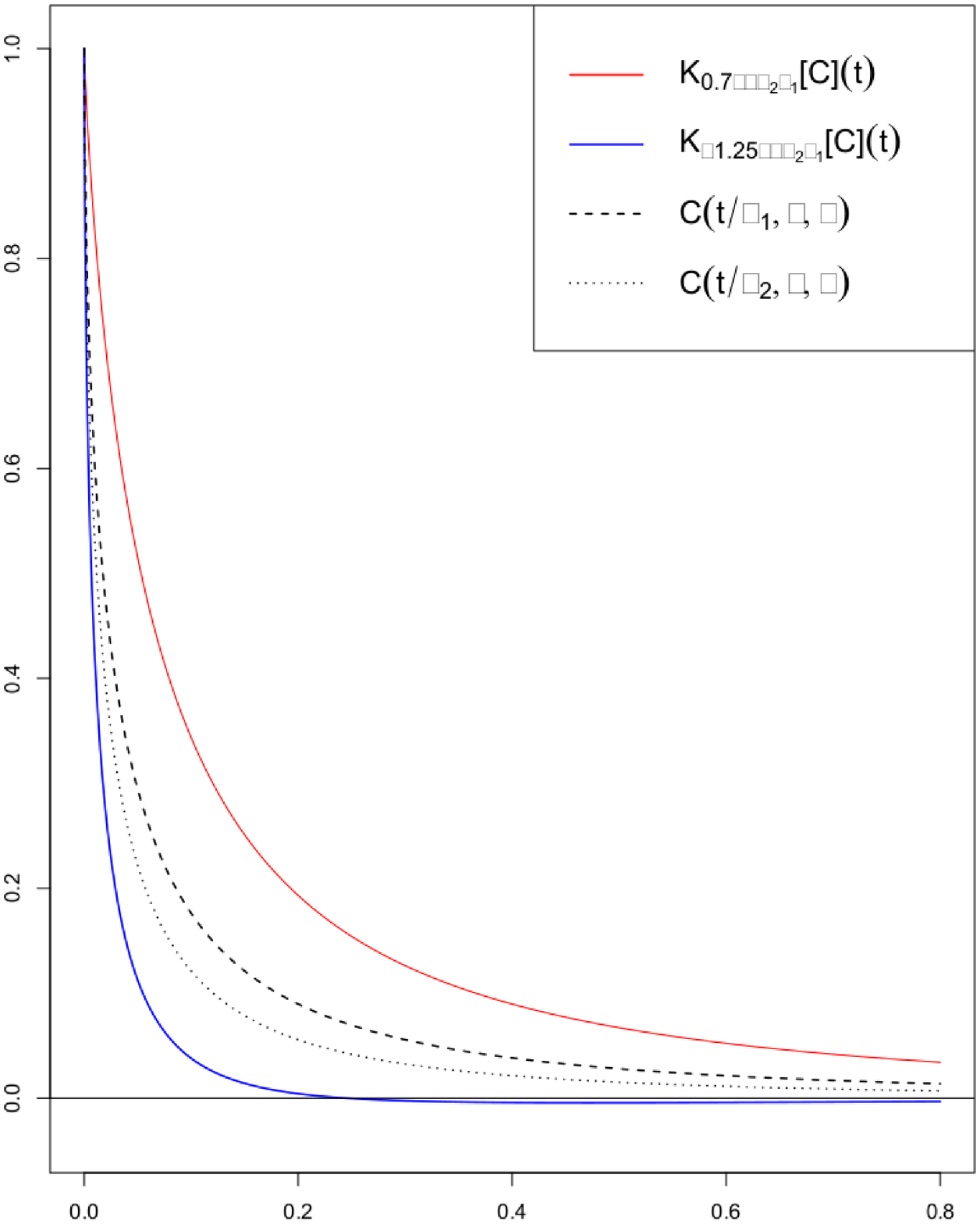}& \includegraphics[width=5.2cm, height=6.5cm]{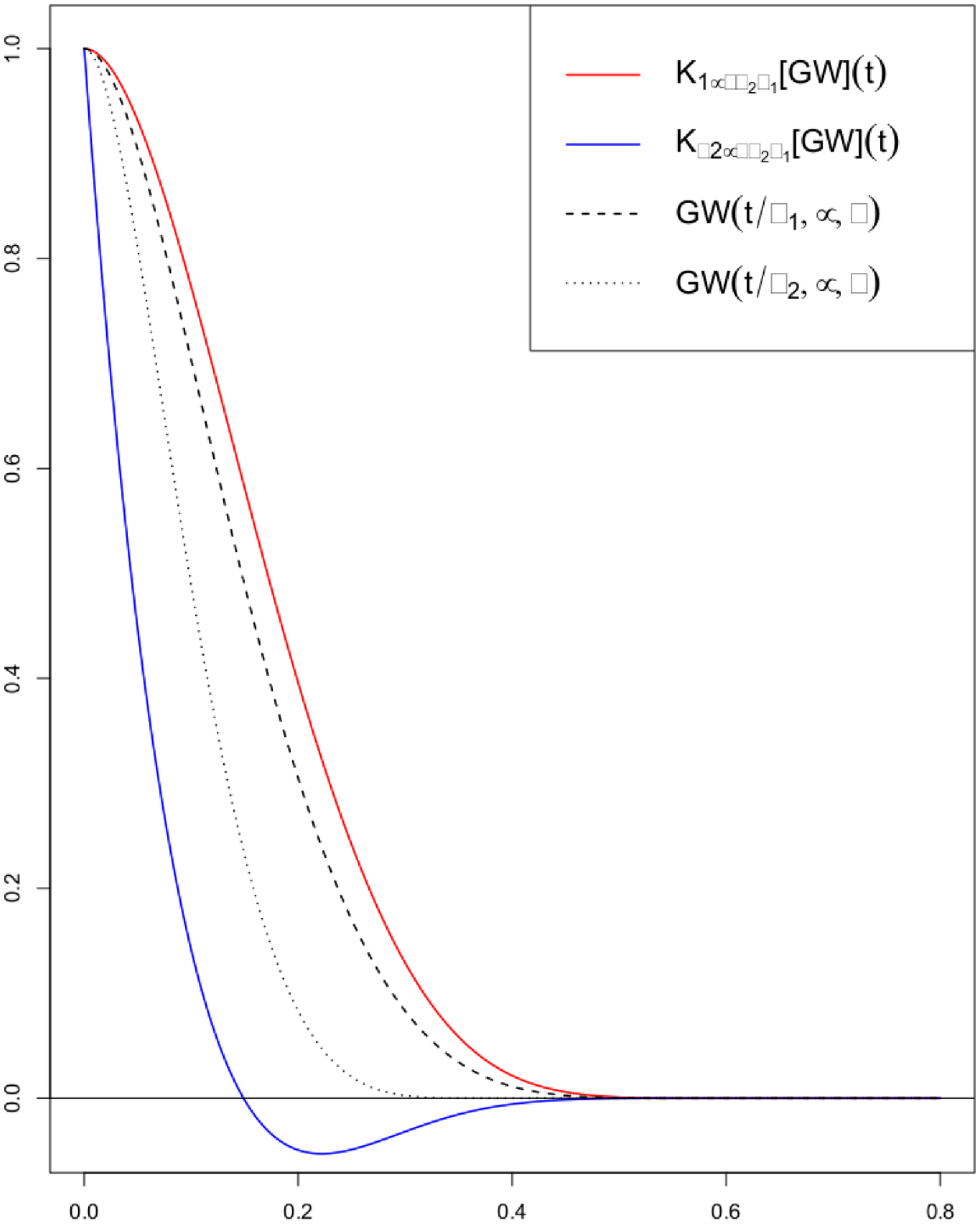}    \\
   (A)&(B)&(C)\\
\end{tabular}
\caption{From left to right: (A)  comparison  of $K_{\varepsilon;\nu;\beta_2,\beta_1}[{\cal M}](t)$
when $\varepsilon=1$ (red line) and  $\varepsilon=-2$ (blue line) with ${\cal M}(t/\beta_1;\nu)$ and
 ${\cal M}(t/\beta_2;\nu)$  when $\nu=0.5$, $\beta_1=0.075$, $\beta_2=0.15$.
(B) comparison  of $K_{\varepsilon;\delta,\lambda;\beta_2,\beta_1}[{\cal C}](t)$
when $\varepsilon=0.7$ (red line) and  $\varepsilon=-1.25$ (blue line) with ${\cal C}(t/\beta_1;\delta,\lambda)$ and
 ${\cal C}(t/\beta_2;\delta,\lambda)$  when $\delta=0.6$, $\lambda=2.5$,  $\beta_1=0.2$, $\beta_2=0.3$.
 (C) comparison  of $K_{\varepsilon;\mu,\kappa;\beta_2,\beta_1}[{\cal GW}](t)$
when $\varepsilon=1$ (red line) and  $\varepsilon=-2$ (blue line) with ${\cal GW}(t/\beta_1;\mu,\kappa)$ and
${\cal GW}(t/\beta_2;\mu,\kappa)$  when $\mu=4.5$, $\kappa=0$,  $\beta_1=0.4$, $\beta_2=0.6$.
\label{fig:fcovt222}}
\end{figure}

 These three examples show that   operator  (\ref{zastavnyi1}) can change  substantially the features of given families $\phi$
 in terms of both differentiability at the origin and negative correlations.
The solution of Problem \ref{PP}  for the Mat{\'e}rn and Cauchy families, passes necessarily through the specification of the properties of the radial Fourier transforms of the radially symmetric functions ${\cal M}(\| \cdot \|; \nu)$ and ${\cal C}(\|\cdot\|; \delta, \lambda)$ in $\R^d$. For the Mat{\'e}rn family, such a Fourier transform is available in closed form. For the Generalized Cauchy family we obtain the Fourier transform  as series expansions generalizing a  recent result obtained by  \cite{Lim2010}. The plan of the paper is the following: Section 2 contains the necessary preliminaries and background. Section 3 gives the main results of this paper.

\section{Preliminaries}

We start with some expository material. A function $f: \R^d \to \R$ is called positive definite if, for any collection $\{ a_k \}_{k=1}^N \subset \R$ and points $\bx_1, \ldots, \bx_N \in \R^d$, the following holds:
$$ \sum_{k=1}^N \sum_{h=1}^N a_k f(\bx_k-\bx_h ) a_h \ge 0. $$
By Bochner's theorem, continuous positive definite functions are the Fourier transforms of positive and bounded measures, that is
\def\bomega{\boldsymbol{\omega}}
\begin{equation}
f(\bx) = \int_{\R^d } {\rm e}^{\mathsf{i} \langle \bx, \boldsymbol{z} \rangle} F ({\rm d} \boldsymbol{z}),  \qquad \bx \in \R^d,
\end{equation} where $\langle \cdot,\cdot\rangle$ denotes inner product and where $\mathsf{i}$ is the complex number such that $\mathsf{i}^2=-1$.
Additionally, if $f(\bx)=\phi(\|\bx\|)$ for some continuous function defined on the positive real line,  Schoenberg's theorem \citep[][with the references therein]{Dale14} shows that $f$ is positive definite if and only if its {\em radial part} $\phi$  can be written as
\begin{equation}
\label{schoenberg}  \phi(t) = \int_{[0,\infty)} \Omega_d(\xi t) G_d({\rm d} \xi),  \qquad t \ge 0,
\end{equation}
where $G_d$ is a positive and bounded measure, and where $$ \Omega_{d}(t) = t^{-(d-2)/2} J_{(d-2)/2}(t), \qquad t \ge 0, $$
where $J_{\nu}$ defines a Bessel function of order $\nu$.  If $\phi(0)=1$, then $G_d$ is a probability measure \citep{Dale14}. Classical Fourier inversion in concert with Bochner's theorem shows that the function $\phi$ belongs to the class $\Phi_d$ if and only if it admits the representation (\ref{schoenberg}), and this in turn happens if and only if the function $\widehat{\phi}_d: [0,\infty) \to  [0,\infty)$, defined through
\begin{equation} \label{FT}
 \widehat{\phi}_d(z):= {\cal F}_d[\phi(t)](z)= \frac{z^{1-d/2}}{(2 \pi)^d} \int_{0}^{\infty} t^{d/2} J_{d/2-1}(tz)  \phi(t) {\rm d} t, \qquad z \ge 0,
  \end{equation}
is nonnegative and such that $\int_{[0,\infty)} \widehat{\phi}_d(z) z^{d-1} {\rm d} t < \infty$. Note that we intentionally put a subscript $d$ into $G_d$ and $\widehat{\phi}_d$ to emphasize the dependence on the dimension $d$ corresponding to the class $\Phi_d$ where $\phi$ is defined. This is explicitly stated in \cite{Dale14}, where it is explained that for any member of the class $\Phi_d$ there exists at least $G_1, \ldots, G_d$ non negative bounded measures in the representation (\ref{schoenberg}). Hence the term $d$-Schoenberg measures proposed therein. Finally, a convergence argument as much as in Schoenberg \citep{Sc38} shows that $\phi \in \Phi_{\infty}$ if and only if
\begin{equation}\label{Boch}
\phi(\sqrt{t}) = \int_{[0,\infty)} {\rm e}^{-\xi t } G({\rm d} \xi), \qquad t \ge 0,
\end{equation}
for $G$ positive, nondecreasing and bounded. Thus, $\psi:=\phi(\sqrt{.})$ is the Laplace transform of $G$, which shows, in concert with Bernstein's theorem, that the function $\psi$ is completely monotonic on the positive real line, that is $\psi$ is infinitely often differentiable and $(-1)^k \psi^{(k)}(t) \ge 0$, $t>0$, $k \in \mathbb{N}$. Here, $\psi^{(k)}$ denotes the $k$-th order derivative of $\psi$, with $\psi^{(0)}=\psi $ by abuse of notation.

For a given $d \in \mathbb{N}$, direct inspection shows that for any scale parameter $\beta>0$, the Fourier transform (\ref{FT}) of $\phi(\cdot/\beta)$ is identically equal to $\beta^{d} \widehat{\phi}_{d}( \beta \cdot)=:\widehat{\phi}_{d,\beta}(\cdot)$, and we shall repeatedly make use of this fact for the results following subsequently.
It is well known that the Fourier transform of the Mat\'ern covariance function ${\cal M}$ in Equation (\ref{matern}) is given by  \citep{Stein:1999}
\begin{equation} \label{stein10}
\widehat{{\cal M}}_{d,\beta}(z;\nu)= \frac{\Gamma(\nu+d/2)}{\pi^{d/2} \Gamma(\nu)}
\frac{\beta^d}{(1+\beta^2z^2)^{\nu+d/2}}
, \qquad z \ge 0.
\end{equation}


An ingenious approach in \cite{LiTe09} shows that the Fourier transform of Generalized Cauchy covariance function ${\cal C}$ in Equation (\ref{cauchy}) can be written as
\begin{equation} \label{stein1}
\widehat{{\cal C}}_{d,\beta}(z;\delta,\lambda) =-\frac{
\beta^{d/2+1}z^{-d}}{2^{d/2-1}\pi^{d/2+1}}\Im\left(\int_0^\infty\frac{\mathcal{K}_{(d-2)/2}(
\beta t)}{(1+e^{i\frac{\pi\delta}{2}}(t/z)^{\delta})^{\lambda/\delta}}t^{d/2} {\rm d}t\right), \qquad z > 0,
\end{equation} where $\Im$ denotes the imaginary part of a complex argument,  ${\cal K}_{(d-2)/2}$ is the modified Bessel function of the second kind of order $(d-2)/2$ and $\beta$ is a scale parameter. A closed  form expression for $\widehat{{\cal C}}_{d,\beta}$ has been elusive for longtime.  \cite{Lim2010} shed some light for this problem giving a infinite series representation of the spectral density
under some specific restriction of the parameters.
The result following below
 shows that the series representation given in   \cite{Lim2010} is valid without any restriction on the parameters.

\begin{thm}\label{the3}
 Let ${\cal C}(\cdot;\delta,\lambda)$ be the Generalized Cauchy covariance function
 as defined in Equation (\ref{cauchy}).
   Then, it is true that
 \begin{equation}\label{spectrGW_2}
\begin{aligned}
\widehat{\cal C}_{d,\beta}(z;\delta,\lambda)&=
\frac{z^{-d}}{\pi^{d/2}}\frac{1}{\Gamma(\frac{\lambda}{\delta})} \sum_{n=0}^\infty \frac{(-1)^n}{n!}
\frac{\Gamma(\frac{\lambda}{\delta} + n)\Gamma(d/2-(\frac{\lambda}{\delta}+n)\delta/2)}{\Gamma((\frac{\lambda}{\delta}+n)\delta/2)}\left(\frac{z\beta}{2}\right)^{(\frac{\lambda}{\delta}+n)\delta} \\
&+\frac{z^{-d}}{\pi^{d/2}\delta}\frac{2}{\Gamma(\frac{\lambda}{\delta})} \sum_{n=0}^\infty \frac{(-1)^n}{n!}
\frac{\Gamma\left(\displaystyle{\frac{2n+d}{\delta}} \right)\Gamma\left(\frac{\lambda}{\delta}-  \displaystyle{\frac{2n+d}{\delta}}\right)}{\Gamma(n+d/2)}\left(\frac{z\beta}{2}\right)^{2n+d},
\end{aligned}
\end{equation}
with $z>0$, where $\delta \in (0,2)$ and $\lambda >0$.
\end{thm}

\begin{proof} The Mellin-Barnes transform is defined through the identity
which is given by
\begin{eqnarray}\label{asser1}
\frac{1}{(1+x)^{\alpha}} = \frac{1}{2\pi \mathsf{i}}\frac{1}{\Gamma(\alpha)}
\oint_{\Lambda} ~x^u\Gamma(-u)\Gamma(\alpha+u) \text{d}u,
\end{eqnarray}
here $\Gamma(\cdot)$ denotes the Gamma function. This representation is valid for any $x \in \R$. The contour $\Lambda$ contains the vertical line which passes between left and right poles in the complex plane $u$
from negative to positive imaginary infinity, and should be closed to the left in case $x>1$, and to the right complex infinity if $0<x<1.$ \\
We now proceed to compute $\widehat{{\cal C}}_{d,\beta}$ as follows:
\begin{equation*}
\begin{aligned}
\widehat{{\cal C}}_{d,\beta}(\norm{\boldsymbol{z}};\delta,\lambda) =&
\mathcal{F}_d[\mathcal{C}_{d,\beta}(t;\delta,\lambda)](\norm{\boldsymbol{z}})\\
=&\beta^d \mathcal{F}_d[\mathcal{C}_{d}(t;\delta,\lambda)](\beta \norm{\boldsymbol{z}}), \qquad \boldsymbol{z} \in \R^d.  \\
\end{aligned}
\end{equation*}
Applying Equation (\ref{asser1}), we obtain
\begin{equation*}
\begin{aligned}
\widehat{{\cal C}}_{d,\beta}(\norm{\boldsymbol{z}};\delta,\lambda)=&\frac{\beta^d}{(2\pi)^d}\frac{1}{\Gamma(\frac{\lambda}{\delta})}\int_{\R^d} {\rm e}^{ \mathsf{i} \beta \langle\boldsymbol{z},\boldsymbol{x}\rangle}\oint_{\Lambda} \Gamma(-u) \Gamma(u+\frac{\lambda}{\delta})\norm{\boldsymbol{x}}^{u\delta} \text{d}u~ \text{d} \boldsymbol{x} \\
=&\frac{\beta^d}{(2\pi)^d}\frac{1}{\Gamma(\frac{\lambda}{\delta})}\oint_{\Lambda} \Gamma(-u) \Gamma(u+\frac{\lambda}{\delta})\int_{\R^d} {\rm  e}^{\mathsf{i} \beta \langle\boldsymbol{z}, \boldsymbol{x}\rangle}\norm{\boldsymbol{x}}^{u\delta}\text{d} \boldsymbol{x} \text{d}u~. \\
\end{aligned}
\end{equation*}
We now invoke the well known relationship \citep{allendes2013solution},
\begin{eqnarray*} \int_{\R^d}
{\rm e}^{\mathsf{i} \langle\boldsymbol{z},\boldsymbol{x}\rangle}\norm{\boldsymbol{x}}^{u\delta}\text{d} \boldsymbol{x}=\frac{2^{d+u\delta}\pi^{d/2}\Gamma(d/2+u\delta/2)}{\Gamma(-u\delta/2)\norm{\boldsymbol{z}}^{d
+ u\delta}},
\end{eqnarray*}
and, by abuse of notation, we now write $z:=\norm{\boldsymbol{z}}$. We have
\begin{eqnarray} \label{tonto-lava}
\widehat{\cal C}_{d,\beta}(z;\delta,\lambda)&=&\frac{\beta^d}{(2\pi)^d}\frac{1}{\Gamma(\frac{\lambda}{\delta})}\frac{1}{2\pi \mathsf{i} }\oint_{\Lambda} ~\Gamma(-u) \Gamma(u+\frac{\lambda}{\delta}) \frac{2^{d+u\delta}\pi^{d/2} \Gamma(d/2+u\delta/2)}{\Gamma(-u\delta/2)}\frac{1}{{(\beta z)}^{d + u\delta}} \text{d}u \nonumber \\
&=&\frac{z^{-d}}{\pi^{d/2}}\frac{1}{\Gamma(\frac{\lambda}{\delta})}\frac{1}{2\pi \mathsf{i} } \oint_{\Lambda}  \frac{\Gamma(-u) \Gamma(u +\frac{\lambda}{\delta})\Gamma(d/2+u\delta/2)}{\Gamma(-u\delta/2)}\left(\frac{2}{z\beta}\right)^{u\delta}  \text{d}u.
\end{eqnarray}
For any given value of $|2/z\beta|,$ it does not matter whether it is smaller or greater than $1$. In fact,  we may close the contour to the left complex infinity.
This is a convergent series for any values of the variable $z$ because we have a situation when denominator supresses numerator in the coefficient in front of powers of $z.$
We now observe that the functions $u \mapsto \Gamma(\frac{\lambda}{\delta} + u)$ and $u \mapsto \Gamma\left(d/2+\frac{u\delta}{2}\right)$ contain poles in the complex plane, respectively when $\frac{\lambda}{\delta} + u = -n$, and when $d/2+\frac{u\delta}{2} = -n,$ $ n \in \mathbb{N}$. Using this fact and through direct inspection we obtain that  the right hand side in (\ref{tonto-lava}) matches with (\ref{spectrGW_2}). The proof is completed.
\end{proof}

\section{Main Results}
We start by providing a solution to Problem \ref{PP} when $\phi(\cdot; \btheta)= {\cal M}(\cdot; \nu)$, so that $\boldsymbol{\theta} \equiv \nu$ and $\Theta=(0,\infty)$.
\begin{thm}\label{theo11}
Let   ${\cal M}(\cdot;\nu)$ be the Mat\'ern function as defined in  Equation (\ref{matern}). Let $K_{\varepsilon;\nu;\beta_2,\beta_1}[{\cal M}]$ with  $0<\beta_1<\beta_2$, be the  Zastavnyi operator (\ref{zastavnyi1}) related to the function ${\cal M}(\cdot; \nu)$.
\begin{enumerate}
\item  Let $\varepsilon>0$. Then, $K_{\varepsilon;\nu;\beta_2,\beta_1}[{\cal M}] \in \Phi_{\infty}$
if and only if $\varepsilon\geq 2\nu >0$;
\item For a given $d\in\mathbb{N}$, let $\varepsilon<0$. Then, $K_{\varepsilon;\nu;\beta_2,\beta_1}[{\cal M}] \in\Phi_{d} $ 
if and only if $\varepsilon\leq -d<0$.
\end{enumerate}
\end{thm}


The following result gives some conditions for  the solution of  Problem \ref{PP} when  $\phi(\cdot; \btheta)= {\cal GW}(\cdot;\mu,\kappa)$, so that $\btheta=(\kappa,\mu)^{\top}$, $\Theta=[0,\infty) \times (0,\infty)$.

\begin{thm}\label{theo111}
Let $d$ be a positive integer. Let  ${\cal GW}(\cdot;\mu,\kappa)$ be the function defined through  Equations (\ref{eq:wendland}) and  (\ref{eq:wendland1}), for $\kappa>0$ and $\kappa=0$ respectively. Let $K_{\varepsilon; \mu,\kappa,\beta_2,\beta_1}[{\cal GW}]$ with $0<\beta_1<\beta_2$, be the  Zastavnyi operator (\ref{zastavnyi1})
related to the function ${\cal GW}(\cdot; \mu,\kappa)$. Then:
\begin{enumerate}
\item   If  $\varepsilon\geq 2\kappa+1>0$ and $\mu\geq (d+7)/2 + \kappa$,  then  $K_{\varepsilon; \mu,\kappa,\beta_2,\beta_1}[{\cal GW}] \in \Phi_{d}$.
\item     $K_{\varepsilon; \mu,\kappa,\beta_2,\beta_1}[{\cal GW}] \in \Phi_{d}$ if and only if $\varepsilon= 2\kappa+1$ and $\mu\geq (d+7)/2 + \kappa$.
\item   If  $\varepsilon\leq -d<0$ and $\mu\geq (d+7)/2 + \kappa$,  then  $K_{\varepsilon; \mu,\kappa,\beta_2,\beta_1}[{\cal GW}] \in \Phi_{d} $.
\item   $K_{-d; \mu,\kappa,\beta_2,\beta_1}[{\cal GW}] \in \Phi_{d}$ if and only if    $\mu\geq (d+1)/2 + \kappa$.

\end{enumerate}
\end{thm}

We now assume that for a fixed $d$, $\lambda>d$. This condition is necessary to ensure integrability of Generalized Cauchy covariance functions, and hence the related spectral density to be bounded. The following result provides a solution to Problem \ref{PP} when  $\phi(\cdot; \btheta)= {\cal C}(\cdot; \delta,\lambda)$ for $0<\delta<2$, so that $\Theta=(0,2) \times (d,\infty)$.

\begin{thm}\label{theo22}
 Let ${\cal C}(\cdot;\delta,\lambda)$ for $0<\delta<2$  be the Generalized Cauchy  function as defined in  Equation (\ref{cauchy}). Let $K_{\varepsilon;\delta,\lambda;\beta_2,\beta_1}[{\cal C}]$ with  $0<\beta_1<\beta_2$, be the  Zastavnyi operator (\ref{zastavnyi1}) related to the function ${\cal C}(\cdot; \delta,\lambda)$.
\begin{enumerate}
\item Let $\varepsilon>0$. If $\varepsilon\geq\delta >0$ and  $\delta<1$, then $K_{\varepsilon;\delta,\lambda;\beta_2,\beta_1}[{\cal C}] \in \Phi_{\infty}$;
\item Let $\varepsilon<0$.  $K_{\varepsilon;\delta,\lambda;\beta_2,\beta_1}[{\cal C}] \in \Phi_{\infty}$ if and only if $\varepsilon\leq -\lambda <0$.
\end{enumerate}
\end{thm}

Two technical lemmas are needed prove our   last result.
\begin{lemma}\label{lem10}
Let $\mathcal{K}_{\nu}:[0,\infty)\to \R$ be the  function defined through (\ref{matern}). Let $\nu>0$. Then, for all $z>0$,
\begin{enumerate}
                    \item $\underset{z\to +\infty}{\lim}z\frac{\mathcal{K}_{\nu}^{\prime}(z)}{\mathcal{K}_{\nu}(z)}=-\infty$, for all $\nu\in(-\infty,+\infty)$;
                    \item $\underset{z\to +0}{\lim}z\frac{\mathcal{K}_{\nu}^{\prime}(z)}{\mathcal{K}_{\nu}(z)}=-\nu$,  \qquad for $\nu>1$.
                  \end{enumerate}
\end{lemma}
\begin{proof}
To prove the two assertions, it is enough to use the following result \citep{Baricz11},
\begin{equation}\label{lem100}
-\sqrt{\frac{\nu}{\nu-1}z^2+\nu^2}<\frac{z\mathcal{K}_{\nu}^{\prime}(z)}{\mathcal{K}_{\nu}(z)}<-\sqrt{z^2+\nu^2},
\end{equation}
where the left hand side of (\ref{lem100}) is true for all $\nu>1$, and the right hand side holds for all $\nu\in \R$.
\end{proof}

\begin{lemma}\label{lem11}
  Let ${\cal C}(\cdot;2,\lambda)$ be the Cauchy correlation function as defined at (\ref{cauchy}). Let $\beta>0$. Then, for $d>\lambda/2+2$ and $2\epsilon<-\lambda$ the following assertions are equivalent:
  \begin{enumerate}
    \item $\beta^{\epsilon}\widehat{{\cal C}}_{d,\beta}(z;2,\lambda)$  is decreasing with respect to $\beta$ on $[0,+\infty)$, for every $z,\lambda$;
    \item $\beta^{\varepsilon+\frac{2d+\lambda}{4}}\mathcal{K}_{\frac{2d-\lambda}{4}}(\beta)$ is decreasing with respect to $\beta$ on $[0,+\infty)$, for every $z,\lambda$;
    \item $ (\varepsilon+\frac{2d+\lambda}{4})+\beta\frac{\mathcal{K}_{\frac{2d-\lambda}{4}}^{\prime}(\beta)}{\mathcal{K}_{\frac{2d-\lambda}{4}}(\beta)}<0$, $\beta\in [0,\infty)$.
  \end{enumerate}
\end{lemma}

\begin{proof}
Showing that $\beta^{\varepsilon}\widehat{{\cal C}}_{d,\beta}(z;2,\lambda)$ is decreasing with respect to $\beta$ is the same as showing that
$\beta^{\varepsilon+\lambda/4+d/2}\mathcal{K}_{d/2-\lambda/4}( \beta)$ is decreasing. Point {\it 2} of Lemma~\ref{lem11} holds if and only if
$$ \beta^{\varepsilon+\frac{2d+\lambda}{4}-1}\left(\left (\varepsilon+\frac{2d+\lambda}{4} \right )+\beta\frac{\mathcal{K}_{\frac{2d-\lambda}{4}}^{\prime}(\beta)}{\mathcal{K}_{\frac{2d-\lambda}{4}}(\beta)}\right)<0.$$ Applying  point {\it 2} of Lemma~\ref{lem10} and the fact that $\beta\mathcal{K}_{\frac{2d-\lambda}{4}}^{\prime}(\beta)/\mathcal{K}_{\frac{2d-\lambda}{4}}(\beta)$ is decreasing with respect to $\beta$, the three assertions of Lemma~\ref{lem11} are true if $2d>\lambda+4$ and $2\varepsilon<-\lambda$. The proof is completed.
\end{proof}
We are now able to fix a solution to Problem \ref{PP} when $\phi(\cdot; \boldsymbol{\theta})= {\cal C}(\cdot; 2,\lambda)$, so that $\boldsymbol{\theta}=\nu$ and $\Theta=(0,\infty)$.
\begin{thm}\label{theo33}
Let ${\cal C}(\cdot;2,\lambda)$  be the Cauchy  function as defined in  Equation (\ref{cauchy})
 and let $K_{\varepsilon;2,\lambda;\beta_2,\beta_1}[{\cal C}]$ with  $0<\beta_1<\beta_2$ be the Zastavnyi operator (\ref{zastavnyi1}) related to the function ${\cal C}(\cdot; 2,\lambda)$.
Then, for $\lambda<2d-4$, $K_{\varepsilon;2,\lambda;\beta_2,\beta_1}[{\cal C}] \in \Phi_{d} $ provided  $2\varepsilon < -\lambda$.
\end{thm}

\begin{proof}
We need to find conditions  such that  $K_{\varepsilon;2,\lambda;\beta_2,\beta_1}[{\cal C}] \in  \Phi_{d}$.
This is equivalent to the following condition:
$$\beta_1^{\varepsilon}\widehat{{\cal C}}_{d,\beta_1}(z;2,\lambda)- \beta_2^{\varepsilon}\widehat{{\cal C}}_{d,\beta_2}(z;2,\lambda)\geq 0.$$
Thus, we need to prove that the function $\beta^{\varepsilon}\widehat{{\cal C}}_{d,\beta}(z;2,\lambda)$ is decreasing with respect to $\beta$. Using Lemma \ref{lem11}, we have that $K_{\varepsilon;2,\lambda;\beta_2,\beta_1}[{\cal C}] \in  \Phi_{d}.$
\end{proof}


\section*{Acknowledgements}
The authors dedicate this work to Viktor Zastavnyi for his sixtieth birthday. \\
Partial support was provided   by Millennium
Science Initiative of the Ministry
of Economy, Development, and
Tourism, grant "Millenium
Nucleus Center for the
Discovery of Structures in
Complex Data"
for Moreno Bevilacqua and Emilio Porcu,
by FONDECYT grant 1160280, Chile for Moreno Bevilacqua and
by FONDECYT grant 1130647 , Chile for Emilio Porcu
and by grant
Diubb 170308 3/I from the university of Bio Bio for Tarik Faouzi. Tarik Faouzi and Igor Kondrashuk thank the support of project DIUBB 172409 GI/C at University of B{\'\i}o-B{\'\i}o. The work of I.K. was supported
in part by Fondecyt (Chile) Grants No. $1121030$ and by DIUBB (Chile) Grant No. $181409 3/R$.



\bibliographystyle{apalike}

\bibliography{mybib12m}







\end{document}